\documentclass[11pt]{article}
\usepackage{amssymb,amsthm,amsmath,amsfonts}
\usepackage[latin1]{inputenc}
\usepackage{graphicx}
\usepackage{verbatim}
\usepackage{hyperref}
\usepackage{array}
\usepackage{color}
\newtheorem{Theorem}{Theorem}[section]

\newtheorem{Lemma}[Theorem]{Lemma}

\newtheorem{Remark}[Theorem]{Remark}

\newtheorem*{theorem-a}{Moser, 1971}
\newtheorem*{theorem-b}{Lions, 1985}
\begin{document}

\title{Ground state solution for a Kirchhoff problem with
exponential critical growth
\thanks{Research partially supported by the National Institute of Science and Technology of
Mathematics INCT-Mat, CNPq grants 308339/2010-0 and 454749/2011-2.}}
\author{Giovany M. Figueiredo \\ Universidade Federal do Par\'a,
Faculdade de Matem\'atica, \\ CEP 66075-110, Bel\'em-PA, Brasil  \\
\textsf{giovany@ufpa.br}
 \\ \\ Uberlandio B. Severo \\ Universidade Federal do Para\'{\i}ba, Departamento de Matem\'atica, \\
 CEP 58051-900, Jo\~ao Pessoa-PB, Brasil \\
 \textsf{uberlandio@mat.ufpb.br} }

\date{}

\maketitle{}

\numberwithin{equation}{section}

\begin{abstract}
We establish the existence of a positive ground state solution for
a Kirchhoff problem in $\mathbb{R}^2$ involving critical
exponential growth, that is, the nonlinearity behaves like
$\exp(\alpha_{0}s^{2})$ as $|s| \to \infty$, for some
$\alpha_{0}>0$. In order to obtain our existence result we used
minimax techniques combined with the Trudinger-Moser inequality.
\end{abstract}

\bigskip

\noindent{\small\emph{2010 Mathematics Subject Classification:}
35J20, 35J25, 35J60, 35Q60.

}

\section{Introduction}
This work is concerned with the existence of a positive ground
state solution for a \textit{nonlocal Kirchhoff problem} of the
type
\begin{equation}\label{P}
\left\{
\begin{aligned}
-m(\|u\|^{2})\Delta u &= f(x,u)&\mbox{in}&\quad \Omega, \\
%u&>0&\mbox{in}&\quad \Omega, \\
u&=0\quad &\mbox{on}&\quad\partial\Omega,
\end{aligned}
 \right.
\tag{$P$}
\end{equation}
where $\Omega$ is a smooth bounded domain in $\mathbb{R}^2$,
$\|u\|^{2}:=\int_{\Omega}|\nabla u |^{2}\textrm{d}x$ is the norm of
the gradient in $W^{1,2}_0(\Omega)$, $m:\mathbb{R}_{+}\rightarrow
\mathbb{R}_+$ and $f:\Omega\times\mathbb{R}\rightarrow \mathbb{R}$
are continuous functions that satisfy some appropriate conditions
and
they will be stated later on.\\

Problem \eqref{P} is called \textit{nonlocal} because of the term
$m(\|u\|^{2})$ which implies that the equation in \eqref{P} is no
longer a pointwise identity. As we will see later the presence of
the term $m(\|u\|^{2})$ provokes some mathematical difficulties
which makes the study of such a class of problems particularly
interesting. Moreover, equation \eqref{P} has a physical appeal.
The main motivation to study problem \eqref{P} is due to the work
of Kirchhoff \cite{Kirchhoff} in which, in 1883, he studied the
hyperbolic equation
\begin{equation}\label{parabolica}
\rho \frac{\partial^2u}{\partial
t^2}-\left(\frac{P_0}{h}+\frac{E}{2L}\int_0^L\left|\frac{\partial
u}{\partial
x}\right|^2\textrm{d}x\right)\frac{\partial^2u}{\partial x^2}=0,
\end{equation}
that extends the classical D'Alembert wave equation, by
considering the effects of the changes in the length of the
strings during the vibrations. The parameters in equation
\eqref{parabolica} have the following meanings: $L$ is the length
of the string, $h$ is the area of cross-section, $E$ is the Young
modulus of the material, $\rho$ is the mass density and $P_0$ is
the initial tension. In fact, \eqref{P} can be seen as a
stationary version of the following evolution problem:
\[
\left\{
\begin{aligned}
\frac{\partial^2u}{\partial
t^2}-m(\|u\|^{2})\Delta u &= f(x,u)&\mbox{in}&\quad \Omega\times (0,T), \\
%u&>0&\mbox{in}&\quad \Omega, \\
u&=0\quad &\mbox{on}&\quad\partial\Omega\times (0,T),\\
u(x,0)&=u_0(x)&\mbox{in}&\quad \Omega,\\
\frac{\partial u}{\partial t}(x,0)&=u_1(x)&\mbox{in}&\quad \Omega,
\end{aligned}
 \right.
\]
which have called the attention of several researchers mainly
after the work of Lions \cite{JL Lions}, where a functional
analysis approach was proposed to study it. We mention that
nonlocal problems also appear in other fields, for example,
biological systems where the function $u$ describes a process
which depends on the average of itself (for example, population
density), see for instance \cite{coalves1, coalves2}
and its references.\\

In this paper, we are also interested in a borderline case of the
Sobolev imbedding theorems, commonly known as the Trudinger-Moser
case. When $n=2$, clearly the Sobolev exponent $2^*$ becomes
infinite and $W^{1,2}(\Omega) \hookrightarrow L^q(\Omega)$ for
$1\leq q < \infty$ but $W^{1,2}(\Omega) \not \hookrightarrow
L^\infty(\Omega)$. To fill this gap, at least in the case where
$\Omega$ is a \textit{bounded domain}, using the Dirichlet norm
$\|\nabla u\|_{2}$ (equivalent to the Sobolev norm in
$W_0^{1,2}(\Omega)$) and replacing the target Lebesgue space by an
Orlicz space, N.~Trudinger \cite{Trudinger}  proved that there
exists $\alpha>0$ such that $W_0^{1,2}(\Omega)$ is embedded into
the Orlicz space $L_{\phi_{\alpha}}(\Omega)$ determined by the
Young function $\phi_{\alpha}(t)=\exp(\alpha t^2)-1$. This result
had many generalizations, extensions and applications in recent
years. In the first direction, it was sharpened
 %: the first one was to find best exponents and the second one was to extend for unbounded domains.
by J.~Moser \cite{Moser71}, who found the best exponent $\alpha$
and in particular he proved the following result:
\begin{theorem-a}
There exists a constant $C>0$ so that
\begin{equation}\label{PTM1}
\sup_{u \in W^{1,2}_0(\Omega)\,:\, \|\nabla u\|_{2}\leq
1}\int_{\Omega} \exp(\alpha
 u^2) \textrm{d}x \leq C |\Omega|,\quad \forall \; \alpha \leq 4\pi.
\end{equation}
Moreover, $4\pi$ is the best constant, that is, the supremum in
\eqref{PTM1} is $+\infty$ if $\alpha > 4\pi$.
\end{theorem-a}
Estimate \eqref{PTM1} is now referred as \textit{Trudinger-Moser
inequality} and plays an important role in geometric analysis and
partial differential equations.

On the crucial question of compactness for the imbedding
$W^{1,2}_0(\Omega) \hookrightarrow L_{\phi_{\alpha}}(\Omega)$ with
$\phi_{\alpha}(t)=\exp(\alpha t^2)-1$, P.~-L.~Lions \cite{PL
Lions} proved that except for ``small weak neighborhoods of $0$''
the imbedding is compact and the best constant $4\pi$ may be
improved in a certain sense. More specifically, among other
results, P.~-L.~Lions proved the following:
\begin{theorem-b}
Let $(u_k)$ be a sequence of functions in $W^{1,2}_0(\Omega)$ with
$\|\nabla u_k\|_2 = 1$ such that $u_k \rightharpoonup u \not
\equiv 0$ weakly in $W_0^{1,2}(\Omega)$. Then for  any $0<p <
4\pi/(1 -\|\nabla u\|_2^2)$ we have
\begin{equation}\label{Lions-1}
\sup_{k} \int_\Omega \exp(p u_k^2) \; \textrm{d}x < \infty .
\end{equation}
\end{theorem-b}
It is clear that this result gives more precise information than
\eqref{PTM1} when $u_k \rightharpoonup u$ weakly in
$W^{1,2}_0(\Omega)$ with $u\not\equiv0$ and it will be crucial to
prove our main result. In this context, we are concerned about the
existence of solution for (\ref{P}) when the nonlinearity $f(x,s)$
has the maximal growth on $s$ for which the functional
$\Phi(u):=\int_\Omega F(x,u)\ \textrm{d}x$, where
$F(x,s)=\int^{s}_{0}f(x,t)\textrm{d}t$, can be studied on the
$W^{1,2}_0(\Omega)-$setting. To be more precise, following the
lines of \cite{adimurthi,djairo,djairo1} and motivated by the
Trudinger-Moser inequality (\ref{PTM1}), we say that $f(x,s)$ has
\textit{exponential subcritical growth} at $+\infty$ if
$\lim_{s\rightarrow +\infty}f(x,s)\exp(-\alpha s^2)=0$ for any
$\alpha>0$ and $f(x,s)$ has \textit{exponential critical growth}
at $+\infty$, if there is $\alpha_{0}>0$ such that
$$
\displaystyle\lim_{s\rightarrow +\infty}f(x,s)\exp(-\alpha s^{2})=
\left\{
\begin{aligned}
0,&\quad \forall\,\alpha > \alpha_{0},\\
+\infty,&\quad \forall\,\alpha < \alpha_{0},
\end{aligned}
 \right.\leqno{(c)_{\alpha_{0}}}
$$
uniformly in $x\in\Omega$. We will restrict our discussion for the
case that $f(x,s)$ has exponential critical growth which is more
involved. \\

For ease of reference we state our assumptions on $m$ and $f$ in a
more precise way. For this, we define
$M(t)=\int_0^tm(s)\textrm{d}s$, the primitive of $m$ so that
$M(0)=0$. The hypotheses on the function
$m:\mathbb{R}_{+}\rightarrow \mathbb{R}_+$ are the following:
\begin{itemize}
    \item [$(M_1)$] there exists $m_{0}>0$ such that $m(t) \geq m_{0}$ for all $t\geq
    0$ and
    $$
M(t+s)\geq M(t)+M(s)\quad \forall\, s,t\geq0;
    $$

    \item [$(M_2)$] there exists constants $a_1, a_2>0$ and $t_0>0$ such that for some $\sigma\in \mathbb{R}$
    $$
m(t)\leq a_1+a_2 t^{\sigma},\,\,\, \forall\, t\geq t_0;
    $$
    \item [$(M_3)$] $\dfrac{m(t)}{t}$ is nonincreasing for $t>0$.
\end{itemize}

Note that condition $(M_1)$ is valid whenever $m(0):=m_0>0$ and
$m$ is nondecreasing. A typical example of a function $m$
satisfying the conditions $(M_{1})-(M_{3})$ is given by
$$
m(t)= m_{0} + at,
$$
where $m_{0}>0$ and $a\geq 0$, which is the model considered in
the original Kirchhoff equation \eqref{parabolica}. An another
example is $m(t)=1+\ln(1+t)$.

As a consequence of $(M_3)$ (see proof of Lemma \ref{crescente}),
a straightforward computation shows that
\begin{itemize}
    \item [$(\widehat{M}_3)$] $\displaystyle\frac{1}{2}M(t)-\frac{1}{4}m(t)t$
    is nondecreasing for $t\geq 0$.
\end{itemize}
In particular, one has
\begin{equation}\label{conseq1}
\frac{1}{2}M(t)-\frac{1}{4}m(t)t\geq 0,\,\,\,\forall\ t\geq 0.
\end{equation}

Here, we also require that $f:\Omega\times \mathbb{R}\rightarrow
\mathbb{R}$ is continuous. Since we intend to find positive
solutions, in all this paper let us assume that $f(x,s)=0$ for
$x\in\Omega$ and for $s\leq 0$. Moreover, $f$ satisfies
$(c)_{\alpha_0}$ and the following conditions:
\begin{itemize}
    \item [$(f_1)$] there exist constants $s_0, K_0>0$ such that
$$
F(x,s)\leq
K_0f(x,s)\quad\forall\,(x,s)\in\Omega\times[s_0,+\infty);
$$
    \item [$(f_2)$] for each $x\in \Omega$, $\dfrac{f(x,s)}{s^3}$ is
    increasing for $s>0$;

    \item [$(f_3)$] there exists
    $\beta_{0}>\dfrac{2}{\alpha_{0}d^{2}}m(4\pi/\alpha_{0})$
    such that
$$
\displaystyle\lim_{s\rightarrow
+\infty}\displaystyle\frac{sf(x,s)}{\exp(\alpha_{0}s^{2})}\geq
\beta_0,\quad \textrm{uniformly in }x\in \Omega,
$$
where $d$ is the radius of the largest open ball contained in
$\Omega$.
\end{itemize}

\medskip

\noindent We observe that condition $(f_1)$ implies
\[
 F(x,s)\geq F(x,s_0)\exp[K_0(s-s_0)] \quad
\forall\; (x,s) \in \mathbb{R}^2\times[s_0,+\infty),
\]
which is reasonable for functions $f(x,s)$ behaving as
$\exp(\alpha_0 s^2)$ at infinity. Moreover, from $(f_{1})$, for
each $\theta>0$ there exists $R_\theta>0$ satisfying
\begin{equation}\label{AR}
\theta F(x,s) \leq sf(x,s),
\quad\forall\,(x,s)\in\Omega\times[R_\theta,+\infty).
\end{equation}
We also have that condition $(f_2)$ implies that
\begin{equation}\label{origem}
\lim_{s\rightarrow 0^+}\frac{f(x,s)}{s^\mu}=0\quad
\textrm{uniformly in}\,\, x\in \Omega
\end{equation}
provided that $\mu\in[0,3)$. In particular, we have $f(x,0)=0$ for
each $x\in \Omega$.
% On the other hand, for
%$s\in(0,s_0)$ and integrating $(f_2)$ in $[s,s_0]$ we reach
%\[
%F(x,s)\leq C_0s^4,\quad\forall\, x\in\Omega,
%\]
%where $C_0=\sup\{F(x,s_0)/s_0^4: x\in\Omega\}$. Thus, in
%particular, for any $\mu\in [0,4)$ one has
%\begin{equation}\label{origem}
%\lim_{s\rightarrow 0^+}\frac{F(x,s)}{s^\mu}=0\quad
%\textrm{uniformly in}\,\, x\in \Omega.
%\end{equation}
As we will see later on, hypothesis $(f_3)$ is necessary to obtain
precise information about the minimax level of the energy
functional associated to problem (\ref{P}).

Generally, the main difficulty encountered in nonlocal Kirchhoff
problems is the competition that there is between the growths of
$m$ and $f$. To overcome this trouble, the authors usually assume
that $m$ is increasing or bounded, as we can see in
\cite{coalves1, coalves3, Bitao, He, Figueiredo2, Li, Liu, Wang}.
We point out that in our arguments we do not suppose that $m$ is
increasing and not bounded above. This allows us to consider the
case $m(t)\equiv 1$ that corresponds to the Dirichlet problem
\[
\left\{
\begin{aligned}
-\Delta u &= f(x,u)&\mbox{in}&\quad \Omega, \\
%u&>0&\mbox{in}&\quad \Omega, \\
u&=0\quad &\mbox{on}&\quad\partial\Omega.
\end{aligned}
 \right.
\]
Furthermore, for the authors knowledge, at the present time, there
is no nonlocal problem involving critical growth of
Trudinger-Moser type. For problems related to \eqref{P}, involving
critical growth in the Sobolev case, we refer the papers
\cite{coalves3, giovany, Wang}.\\

We say that $u\in W^{1,2}_0(\Omega)$ is a weak solution of
(\ref{P}) if holds
\[
m(\|u\|^{2})\int_{\Omega}\nabla u \nabla \phi \ \textrm{d}x
=\int_{\Omega} f(x,u)\phi \ \textrm{d}x,\quad \forall\, \phi\in
W^{1,2}_0(\Omega).
\]
Since $f(x,0)=0$, $u\equiv 0$ is the trivial solution for
(\ref{P}). Thus, our aim is to obtain a nontrivial solution. The
term \textit{ground state} refers to minimizers of the
corresponding energy within the set of nontrivial solutions (see
Section 2). Now, the main result of this work can state as
follows.

\noindent\begin{Theorem}\label{teorema1} Suppose
$(M_{1})-(M_{3})$, $(c)_{\alpha_{0}}$ and $(f_{1})-(f_{3})$ are
satisfied. Then, problem (\ref{P}) has a positive ground state
solution.
\end{Theorem}

An example of a function $f$ satisfying the conditions in Theorem
\ref{teorema1}, with $\alpha_0=1$, is given by
$$
F(x,s)=\frac{s^4}{4}+s^2[\exp(s^{2})-1],\,\, s\geq 0.
$$
Indeed, deriving we get
$$
f(x,s)=s^3+2s[\exp(s^{2})-1]+2s^3\exp(s^{2}),
$$
from which one has $f(x,s)/s^3$ is increasing for $s>0$. A simple
computation shows that
$$
\lim_{s \rightarrow
+\infty}\frac{F(x,s)}{f(x,s)}=0\quad\textrm{and}\quad\lim_{s
\rightarrow +\infty}\frac{sf(x,s)}{\exp(s^{2})}=+\infty,
$$
uniformly in $x\in \Omega$ and so $(f_1)-(f_3)$ are satisfied.

\bigskip

The paper is organized as follows. In Section 2 we present the
variational setting in which our problem will be treated. Section
3 is devoted to show that the energy functional has the mountain
pass geometry and in Section 4 we obtain an estimate for the
minimax level associated to the our functional. Finally, we prove
Theorem \ref{teorema1} in Section 5.

Hereafter, $C$, $C_0$, $C_{1}$, $C_{2}$, ... will denote positive
(possibly different) constants. We shall use the notation
$\|\cdot\|$ for the norm of the gradient in the Sobolev space
$W^{1,2}_0(\Omega)$ and $\|\cdot\|_p$ for the norm in the Leabegue
space $L^p(\Omega)$, $1\leq p<\infty$. The abbreviation a.e.
will mean \textit{almost everywhere}.\\

\section{The variational framework} \label{var}

As we are interested in positive solution, from now on we shall
assume $f(x,s)=0$ for $x\in \Omega$ and $s\leq0$. Since $f(x,s)$
is continuous and satisfies $(c)_{\alpha_{0}}$ and (\ref{origem}),
for $\varepsilon>0$, $\alpha>\alpha_{0}$ and $q\geq 0$, there
exists $C=C(\varepsilon,\alpha, q)>0$ such that
\begin{equation}\label{estimativa 1}
|F(x,s)|\leq\varepsilon s^2+C|s|^q\exp(\alpha s^2)\quad\forall\,
(x,s)\in\Omega\times \mathbb{R}.
\end{equation}
This together with (\ref{PTM1}) yields $F(\cdot,u)\in L^1(\Omega)$
for all $u\in W^{1,2}_0(\Omega)$. Consequently, the functional
\[
I(u) := \frac{1}{2} M(\|u\|^{2})-\int_{\Omega} F(x,u)\ \textrm{d}x
\]
is well defined on $W^{1,2}_{0}(\Omega)$. Moreover, by standard
arguments, $I\in C^1(W^{1,2}_0(\Omega),\mathbb{R})$ with
\[
\langle I'(u),\phi\rangle = m(\|u\|^{2})\int_{\Omega}\nabla u
\nabla \phi \ \textrm{d}x -\int_{\Omega} f(x,u)\phi \
\textrm{d}x,\quad u,\phi\in W^{1,2}_0(\Omega).
\]
Hence, its critical points correspond to weak solutions of
\eqref{P}, that is, $I$ is the Euler-Lagrange functional
associated to \eqref{P}. We are interested in ground state
solution $u$ for \eqref{P} in the following sense: $u$ is positive
and minimizes the energy functional $I$ within the set of
nontrivial solutions of \eqref{P}.

\section{Mountain pass structure}

In order to achieve our existence result, we shall use the
following version of the mountain pass theorem due to Ambrosetti
and Rabinowitz \cite{Ambrosetti}, without the Palais-Smale
condition:

\begin{Theorem}\label{MPT}
Let $E$ be a Banach space and $\Phi \in C^1(E;\mathbb{R})$ with
$\Phi(0)=0$. Suppose that there exist $\rho,\tau>0$ and $e\in E$
such
\begin{equation}\label{GMP}
\inf_{\|u\|=\rho}\Phi(u)\geq \tau\quad \text{and}\quad \Phi(e)\leq
0.
\end{equation}
Then $\Phi$ possesses a Palais-Smale sequence at level $c$
characterized as
\[
c:=\inf_{\gamma\in\Gamma}\max_{t\in [0,1]}\Phi(\gamma(t))\geq\tau,
\]
where $\Gamma = \left\{ \gamma \in C([0,1]; E):\gamma(0)=0\;\
\text{and} \;\ \gamma(1)=e \right\}$.
\end{Theorem}

The number $c$ is called \textit{mountain pass level} or
\textit{minimax level}  of the functional $\Phi$.

In the sequel, we show that the functional $I$ has the mountain
pass geometry, condition (\ref{GMP}) above. This is proved in the
next lemmas:

\begin{Lemma}\label{geometria1}
Assume that conditions $(M_{1})$, $(c)_{\alpha_0}$ and $(f_{2})$
hold. Then, there exist positive numbers $\rho$ and $\tau $ such
that
$$
I(u)\geq \tau,\,\, \forall \,u \in W^{1,2}_{0}(\Omega)\quad
\text{with}\quad \|u\|=\rho.
$$
\end{Lemma}
\begin{proof} By using (\ref{estimativa 1}), we get
$$
\int_{\Omega}F(x,u) \ \textrm{d}x\leq
\varepsilon\int_{\Omega}|u|^{2} \ \textrm{d}x +
C\int_{\Omega}|u|^{q}\exp(\alpha u^{2}) \ \textrm{d}x,\,\, u\in
W^{1,2}_{0}(\Omega).
$$
Here, let us consider $q>2$. From Sobolev imbedding and H\"{o}lder
inequality, for $\|u\|\leq \rho_1$ we reach
$$
\begin{aligned}
\int_{\Omega}F(x,u) \ \textrm{d}x&\leq\varepsilon C_1\|u\|^2+
C\|u\|^{q}_{2q}\left[\int_{\Omega}\exp\left(2\alpha
\|u\|^{2}(u/\|u\|)^{2}\, \right) \textrm{d}x\right]^{1/2}\\
&\leq \varepsilon
C_1\|u\|^2+C_2\|u\|^q\left[\int_{\Omega}\exp\left(2\alpha\rho_1^2
(u/\|u\|)^{2}\, \right) \textrm{d}x\right]^{1/2}.
\end{aligned}
$$
Thus, if $\rho_1\leq \sqrt{2\pi/\alpha}$, using the
Trudinger-Moser inequality (\ref{PTM1}) and condition $(M_{1})$
one has
$$
I(u) \geq \|u\|^{2}\left(\frac{m_{0}}{2}-\varepsilon
C_1-C_3\|u\|^{q-2}\right).
$$
Next, fix $\varepsilon>0$ so that $m_0/2-\varepsilon C_1>0$.
Hence, since $q>2$, choose $0<\rho\leq \rho_1$ verifying
$m_{0}/2-\varepsilon C_1-C_3\rho^{q-2}>0$. Consequently, if
$\|u\|=\rho$ then $I(u)\geq \tau$, where
$$
\tau:=\rho^{2}\left(\frac{m_{0}}{2}-\varepsilon
C_1-C_3\rho^{q-2}\right)>0,
$$
and the lemma is proved.
\end{proof}

\begin{Lemma}\label{geometria2}
Assume that conditions $(M_2)$ and $(f_{1})$ hold. Then, there
exists $e\in W^{1,2}_{0}(\Omega)$ with
 $I(e)<0$ and $\|e\|>\rho$.
\end{Lemma}
\begin{proof}
First, we observe that for all $t\geq t_0$ condition $(M_2)$
implies that
\begin{equation}\label{estimativa M}
M(t)\leq \left\{
\begin{aligned}
&a_0+a_1t+\frac{a_2}{\sigma+1}t^{\sigma+1},\,\, &\textrm{if}&\,\,\, \sigma\neq -1,\\
&b_0+a_1t+a_2\ln t,\,\, &\textrm{if}&\,\,\, \sigma=-1,
\end{aligned}
 \right.
\end{equation}
where
$a_0=\int_0^{t_0}m(t)\textrm{d}t-a_1t_0-a_2t_0^{\sigma+1}/(\sigma+1)$
and $b_0=\int_0^{t_0}m(t)\textrm{d}t-a_1t_0-a_2\ln t_0$. On the
other hand, taking $\theta>\max\{2,2\sigma+2\}$ and using
(\ref{AR}) one can see that there exist constants $C_1, C_2>0$
such that
\begin{equation}
\label{estimativa F}
F(x,s)\geq C_1s^{\theta}-C_2\quad
\forall\,(x,s)\in\Omega\times [0,+\infty).
\end{equation}
Now, choose arbitrarily $u_{0}\in W^{1,2}_{0}(\Omega)$ with
$u_{0}\geq 0$ in $\Omega$ and $\|u_{0}\|=1$. Thus, from
\eqref{estimativa M} and \eqref{estimativa F}, for all $t\geq t_0$
we reach
$$
I(tu_0)\leq\left\{
\begin{aligned}
\frac{a_0}{2}+\frac{a_1}{2}t^2+\frac{a_2}{2\sigma+2}t^{2\sigma+2}-C_1\|u_0\|_{\theta}^{\theta}t^{\theta}+C_2|\Omega|,&\,\, &\textrm{if}&\,\,\, \sigma\neq -1,\\
\frac{b_0}{2}+\frac{a_1}{2}t^2+\frac{a_2}{2}\ln
t-C_1\|u_0\|_{\theta}^{\theta}t^{\theta}+C_2|\Omega|,& \,\,
&\textrm{if}&\,\,\, \sigma=-1,
\end{aligned}
 \right.
$$
from which we conclude that $I(tu_0)\rightarrow -\infty$ as
$t\rightarrow+\infty$ provided that $\theta>\max\{2,2\sigma+2\}$.
Hence, the result follows by considering $e=t_{*}u_{0}$ for some
$t_{*}>0$ enough large.
\end{proof}

\section{Minimax estimates}
According to Lemmas \ref{geometria1} and \ref{geometria2}, let be
$$
c_{*} = \displaystyle\inf_{\gamma \in \Upsilon}
\displaystyle\max_{t \in [0,1]} I(\gamma(t))>0,
$$
the minimax level of $I$, where $ \Upsilon = \{ \gamma \in
C([0,1],W^{1,2}_{0}(\Omega)) : \gamma(0)=0, ~I(\gamma(1)) < 0\}$.\\
In order to get a more precise information about the minimax level
$c_{*}$ obtained by Theorem \ref{MPT}, let us consider the
following sequence $\widetilde{G}_n:\mathbb{R}^2\rightarrow
\mathbb{R}$ of scaled and truncated Green's functions and also
considered by Moser (see \cite{djairo1}):
\begin{equation}\label{MM}
\widetilde{G}_{n}(x)=\frac{1}{\sqrt{2\pi}}\left\{
\begin{aligned}
&(\log n)^{1/2}, &\mbox{if}& \ \ |x|\leq \frac{1}{n}\\
&\dfrac{\log \frac{1}{|x|}}{(\log n)^{1/2}}, &\mbox{if}& \ \ \frac{1}{n}\leq |x|\leq 1\\
&0, &\mbox{if}& \ \ |x|\geq 1.
\end{aligned}
 \right.
\end{equation}
Let $x_{0} \in \Omega$  be such that the open ball $B_{d}(x_{0})$
is contained in $\Omega$, where $d$ was given in $(f_{3})$. It is
standard verify that the functions
$$
G_{n}(x):=\widetilde{G}_{n}\left(\frac{x-x_{0}}{d}\right),\,\,
x\in \Omega,
$$
belongs to $H^{1}_{0}(\Omega)$, $\|G_{n}\|=1$ and the support of
$G_{n}$ is contained in $B_{d}(x_{0})$. Furthermore, we have
\begin{Lemma}\label{limite}
The following inequality holds
\[
\liminf_{n\rightarrow\infty}\int_{B_d(x_0)}\exp(4\pi
G_n^2)\textrm{d}x\geq 3\pi d^2.
\]
\end{Lemma}
\begin{proof}
By change of variable and using the definition of
$\widetilde{G}_n$, we have
\begin{equation}\label{change}
\begin{aligned}
\int_{B_d(x_0)}\exp(4\pi
G_n^2)\textrm{d}x&=d^2\int_{B_{\frac{1}{n}}(0)}\exp(4\pi
\widetilde{G}_n^2)\
\textrm{d}y+d^2\int_{\frac{1}{n}\leq|y|<1}\exp(4\pi
\widetilde{G}_n^2)\ \textrm{d}y\\
&=\pi d^2+2\pi
d^2\int_{\frac{1}{n}}^1\exp\left[\frac{2(\log(1/r))^2}{\log
n}\right]r \textrm{d}r\\
&=\pi d^2+2\pi d^2\log n\int_0^1n^{2s^2-2s}\ \textrm{d}s,
\end{aligned}
\end{equation}
where we also have used the change of variable $s=\log(1/r)/\log
n$ in the last integral. Next, since
$$
2s^2-2s\geq \left\{
\begin{aligned}
-2s,&\,\, &\textrm{for}&\,\,\, s\in[0,1/2],\\
2s-2,&\,\, &\textrm{for}&\,\,\, s\in[1/2,1],
\end{aligned}
 \right.
$$
we get
$$
\begin{aligned}
\log n\int_0^1n^{2s^2-2s}\ \textrm{d}s&\geq \log
n\int_0^{\frac{1}{2}}n^{-2s}\ \textrm{d}s+\log
n\int_{\frac{1}{2}}^1n^{2s-2}\ \textrm{d}s\\
&=1-\frac{1}{n}.
\end{aligned}
$$
Using this estimate in (\ref{change}) and passing to the limit, we
obtain the desired inequality.
\end{proof}

\medskip

Finally, the next result provides the desired estimate for the
level $c_*$.

\begin{Lemma}\label{nivel}
If conditions $(M_{1})-(M_{2})$ and $(f_{3})$ hold, then
$$
c_{*}<\frac{1}{2}M\left(\frac{4\pi}{\alpha_{0}}\right).
$$
\end{Lemma}
\begin{proof} Since $G_n\geq 0$ in $\Omega$ and $\|G_{n}\|=1$, as in the proof of Lemma \ref{geometria2},
we have that $I(tG_n)\rightarrow -\infty$ as
$t\rightarrow+\infty$. Consequently,
$$
c_{*}\leq \displaystyle\max_{t>0}I(t G_{n}),\,\, \forall\ n\in
\mathbb{N}.
$$
Thus, it suffices to show that $\max_{t>0}I(t
G_{n})<\frac{1}{2}M(4\pi/\alpha_{0})$ for some $n\in \mathbb{N}$.
Suppose, by contradiction, that
\begin{equation}\label{gugu}
\displaystyle\max_{t>0}I(t G_{n})\geq
\frac{1}{2}M\left(\frac{4\pi}{\alpha_{0}}\right),\,\, \forall\
n\in \mathbb{N}.
\end{equation}
As $I$ possesses the mountain pass geometry, for each $n$ there
exists $t_{n}>0$ such that
$$
I(t_{n}G_{n})=\max_{t>0}I(t G_{n}).
$$
From this and using that $F(x,s)\geq 0$ for all $(x,s) \in
\Omega\times \mathbb{R}$ by \eqref{gugu} one has $
M(t_{n}^{2})\geq M(4\pi/\alpha_{0})$. By condition $(M_{1})$,
$M:[0,+\infty)\rightarrow [0,+\infty)$ is a increasing bijection
and so
\begin{eqnarray}\label{usar}
t_{n}^{2} \geq \frac{4 \pi}{\alpha_{0}}.
\end{eqnarray}
On the other hand,
$$
\frac{d}{d t}I(tG_{n})_{\displaystyle\bigl|_{t=t_{n}}}=0,
$$
from which we obtain
\begin{equation}\label{des2}
m(t_{n}^{2})
t_{n}^{2}=\displaystyle\int_{\Omega}f(x,t_{n}G_{n})t_{n}G_{n}\
\textrm{d}x\geq
\displaystyle\int_{B_{d}(x_{0})}f(x,t_{n}G_{n})t_{n}G_{n}
 \ \textrm{d}x.
\end{equation}
By change of variable,
$$
\begin{aligned}
m(t_{n}^{2}) t_{n}^{2}&\geq d^{2}
\int_{B_{1}(0)}f(x_{0}+dx,t_{n}\widetilde{G}_{n})t_{n}\widetilde{G}_{n}\
\textrm{d}x\\
&\geq d^{2}
\int_{B_{1/n}(0)}f\left(x_{0}+dx,\frac{t_{n}}{\sqrt{2\pi}}(\log
n)^{1/2}\right)\frac{t_{n}}{\sqrt{2\pi}}(\log n)^{1/2} \
\textrm{d}x.
\end{aligned}
$$
In view of (\ref{usar}), it follows that $(\log
n)^{1/2}t_{n}/\sqrt{2\pi}\rightarrow +\infty$ as
$n\rightarrow\infty$. Hence, by $(f_{3})$ given $\delta>0$ there
exists $s_\delta>0$ such that
\begin{equation}\label{des3}
f(x,s)s\geq (\beta_0-\delta)\exp(\alpha_0s^2),\quad\forall\
(x,s)\in\Omega\times[s_\delta,+\infty).
\end{equation}
So we obtain $n_{0} \in \mathbb{N}$ such that
$$
f\left(x_{0}+dx,\frac{t_{n}}{\sqrt{2\pi}}(\log
n)^{1/2}\right)\frac{t_{n}}{\sqrt{2\pi}}(\log n)^{1/2} \geq
(\beta_{0}-\delta)\exp\left(\alpha_{0}\frac{t_{n}^{2}}{2\pi}\log n
\right),
$$
for all $ n\geq n_{0}$.
Thus,
\begin{eqnarray}\label{usar1}
 m(t_{n}^{2}) t_{n}^{2}&\geq &
(\beta_{0}-\delta)d^{2}\exp\left(\alpha_{0}\frac{t_{n}^{2}}{2\pi}\log
n
\right)\frac{\pi}{n^{2}}\nonumber\\
&=& (\beta_{0}-\delta)\pi d^{2}\exp(-2\log
n)\exp\left(\alpha_{0}\frac{t_{n}^{2}}{2\pi}\log n \right)\nonumber\\
&=& (\beta_{0}-\delta)\pi
d^{2}\exp\left[2(\alpha_{0}\frac{t_{n}^{2}}{4\pi}-1)\log n
\right].
\end{eqnarray}
Note that, from $(M_{2})$, we can conclude that
$$
\displaystyle\frac{m(t_{n}^{2})
t_{n}^{2}}{\exp\left[2(\alpha_{0}\frac{t_{n}^{2}}{4\pi}-1)\log n
\right]}\rightarrow 0\quad\mbox{if}\quad t_{n}\rightarrow +\infty.
$$
Hence, from (\ref{usar1}), $(t_{n})$ must be bounded in
$\mathbb{R}$. So, up to a subsequence, $t_{n}\rightarrow t_{0}\geq
\sqrt{4\pi/\alpha_0}$. Moreover, using (\ref{usar1}) again, we
must have $\alpha_{0}\frac{t_{0}^{2}}{4\pi}-1\leq 0$ and therefore
\begin{eqnarray}\label{usar2}
t_{n}^{2}\rightarrow \frac{4 \pi}{\alpha_{0}}.
\end{eqnarray}
At this point, following arguments as in \cite{djairo} and
\cite{djairo1} we are going to estimate (\ref{des2}) more exactly.
For this, in view of (\ref{des3}), for $0<\delta<\beta_0$ and
$n\in \mathbb{N}$ we set
$$
D_{n,\delta}:=\{x\in B_d(x_0):t_nG_n(x)\geq s_\delta\}\quad
\textrm{and}\quad E_{n,\delta}:=B_d(x_0)\backslash D_{n,\delta}.
$$
Thus, by splitting the integral (\ref{des2}) on $D_{n,\delta}$ and
$E_{n,\delta}$ and using (\ref{des3}), it follows that
\begin{equation}\label{des4}
\begin{aligned}
m(t_n^2)t_n^2\geq&(\beta_0-\delta)\int_{B_d(x_0)}\exp(\alpha_0t_n^2G_n^2)\textrm{d}x
-(\beta_0-\delta)\int_{E_{n,\delta}}\exp(\alpha_0t_n^2G_n^2)\textrm{d}x\\
&+\int_{E_{n,\delta}}f(x,t_nG_n)t_nG_n\ \textrm{d}x.
\end{aligned}
\end{equation}
Since $G_n(x)\rightarrow 0$ for almost everywhere $x\in B_d(x_0)$
we have that the characteristic functions $\chi_{E_{n,\delta}}$
satisfy
$$
\chi_{E_{n,\delta}}\rightarrow 1\,\,\, \textrm{a.e. in
}B_d(x_0)\,\,\, \textrm{as }n\rightarrow\infty.
$$
Moreover, $t_nG_n<s_\delta$ in $E_{n,\delta}$. Thus, invoking the
Lebesgue dominated convergence theorem we obtain
$$
\int_{E_{n,\delta}}\exp(\alpha_0t_n^2G_n^2)\textrm{d}x\rightarrow
\pi d^2\quad \textrm{and}\quad
\int_{E_{n,\delta}}f(x,t_nG_n)t_nG_n\ \textrm{d}x\rightarrow 0.
$$
Now, by using these convergences, (\ref{usar}), (\ref{usar2}) and
Lemma \ref{limite}, passing to the limit in (\ref{des4}) we reach
$$
\begin{aligned}
m\left(\frac{4\pi}{\alpha_{0}}\right)\frac{4\pi}{\alpha_{0}}&\geq(\beta_0-\delta)\liminf_{n\rightarrow\infty}\int_{B_d(x_0)}\exp(4\pi
G_n^2)\textrm{d}x-(\beta_0-\delta)\pi d^2\\
&\geq (\beta_0-\delta)2\pi d^2,\,\,\, \forall\
\delta\in(0,\beta_0),
\end{aligned}
$$
and doing $\delta\rightarrow 0^+$ we get $\beta_{0} \leq
\frac{2}{\alpha_{0}d^{2}}m(4\pi/\alpha_{0})$, which contradicts
$(f_3)$. Thus, the lemma is proved.
\end{proof}

\medskip

At this stage, we consider the Nehari manifold associated to the
functional $I$, namely,
$$
\mathcal{N}:=\{u\in W^{1,2}_0(\Omega):\langle I'(u),u\rangle=0,\,
u\ne 0\}
$$
and the number $b:=\inf_{u\in \mathcal{N}}I(u)$. To compare the
minimax level $c_*$ and $b$, we need the following lemma:
\begin{Lemma}\label{crescente}
If condition $(f_2)$ holds then, for each $x\in \Omega$,
$$
sf(x,s)-4F(x,s)\textrm{ is increasing for }s>0.
$$
In particular, $sf(x,s)-4F(x,s)\geq 0$ for all $(x,s)\in
\Omega\times[0,+\infty)$.
\end{Lemma}
\begin{proof}
Suppose $0<s<t$. For each $x\in \Omega$, we obtain
$$
\begin{aligned}
sf(x,s)-4F(x,s)=\ &\frac{f(x,s)}{s^3}s^4-4F(x,t)+4\int_s^tf(x,\tau)\textrm{d}\tau\\
<\ &\frac{f(x,t)}{t^3}s^4-4F(x,t)+\frac{f(x,t)}{t^3}(t^4-s^4)\\
=\ &tf(x,t)-4F(x,t)
\end{aligned}
$$
and this proves the lemma.
\end{proof}
The next result is crucial in our arguments to prove the existence
of a ground state solution for \eqref{P}.
\begin{Lemma}
If $(M_3)$ and $(f_2)$ are satisfied then $c_*\leq b$.
\end{Lemma}
\begin{proof}
Let $u$ be in $\mathcal{N}$ and define
$h:(0,+\infty)\rightarrow\mathbb{R}$ by $h(t)=I(tu)$. We have that
$h$ is differentiable and
$$
h'(t)=\langle
I'(tu),u\rangle=m(t^2\|u\|^2)t\|u\|^2-\int_{\Omega}f(x,tu)u\
\textrm{d}x,\quad\forall\ t>0.
$$
Since $\langle I'(u),u\rangle=0$, that is,
$m(\|u\|^2)\|u\|^2=\int_{\Omega}f(x,u)u\ \textrm{d}x$, we get
$$
h'(t)=t^3\|u\|^4\left[\frac{m(t^2\|u\|^2)}{t^2\|u\|^2}-\frac{m(\|u\|^2)}{\|u\|^2}\right]
+t^3\int_{\Omega}\left[\frac{f(x,u)}{u^3}-\frac{f(x,tu)}{(tu)^3}\right]u^4\textrm{d}x.
$$
We observe that $h'(1)=0$ and by $(M_3)$ and $(f_2)$, it follows
that $h'(t)\geq 0$ for $0<t<1$ and $h'(t)\leq 0$ for $t>1$. Hence,
$$
I(u)=\max_{t\geq0}I(tu).
$$
Now, defining $g:[0,1]\rightarrow W^{1,2}_0(\Omega)$,
$g(t)=tt_0u$, where $t_0$ is such that $I(t_0u)<0$, we have $g\in
\Upsilon$ and therefore
$$
c_*\leq \max_{t\in[0,1]}I(g(t))\leq \max_{t\geq 0}I(tu)=I(u).
$$
Since $u\in \mathcal{N}$ is arbitrary $c_*\leq b$ and the proof is
complete.
\end{proof}

\begin{Remark}\label{constante}
We observe that if $m(t)\equiv K$ for some $K>0$, then the
arguments in the previous lemma work if we suppose the condition
$f(x,s)/s$ is increasing for $s>0$ holds instead of $(f_2)$. In
this case, we would have $sf(x,s)-2F(x,s)$ increasing for $s>0$ as
the property in Lemma \ref{crescente}.
\end{Remark}

\begin{Remark}\label{ground1} We recall that a solution $u_0$ of \eqref{P} is a
ground state if $I(u_0)=d:=\inf_{u\in \mathcal{S}}I(u)$ where
$$
\mathcal{S}:=\{u\in W^{1,2}_0(\Omega):I'(u)=0,\, u\ne 0\}.
$$
Since $c_*\leq b\leq d$, in order to obtain a ground state $u_0$ for
\eqref{P} it is enough to show that there is $u_0\in \mathcal{S}$
and $I(u_0)=c_*$.
\end{Remark}

\section{Proof of Theorem \ref{teorema1}}

This section is devoted to the proof of our main result. For this
purpose, we shall use the following result of convergence, whose
proof can be found, for instance, in \cite{djairo}:
\begin{Lemma}\label{djairo}
Suppose $\Omega$ is a bounded domain in $\mathbb{R}^2$. Let
$(u_n)$ be in $L^1(\Omega)$ such that $u_n\rightarrow u$ in
$L^1(\Omega)$ and let $f(x, s)$ be a continuous function. Then
$f(x,u_n)\rightarrow f(x, u)$ in $L^1(\Omega)$ provided that $f(x,
u_n)\in L^1(\Omega)$ for all $n$ and $\int_\Omega |f(x,
u_n)u_n|\textrm{d}x \leq C$.
\end{Lemma}

\medskip

\begin{proof}[Proof of Theorem \ref{teorema1}] By Lemmas \ref{geometria1} and \ref{geometria2} we can invoke Theorem \ref{MPT} to obtain a sequence
$(u_{n})$ in  $W^{1,2}_{0}(\Omega)$ verifying
$$
I(u_{n}) \rightarrow c_{*} \,\,\,\mbox{and} \,\,\,\, I'(u_{n})
\rightarrow 0.
$$
By using (\ref{AR}), with $\theta>4$, $(M_1)$ and \eqref{conseq1}
we obtain
\begin{equation}\label{limitada}
\begin{aligned}
C + \|u_{n}\|&\geq  I(u_{n})-
\displaystyle\frac{1}{\theta}\langle I'(u_{n}),u_{n}\rangle\\
& \geq
\left(\frac{\theta-4}{4\theta}\right)m_0\|u_{n}\|^{2}-C_\theta
|\Omega|,\,\,\forall\ n \in \mathbb{N},
\end{aligned}
\end{equation}
for some $C>0$, where $C_\theta=\sup\{|f(x,s)s-\theta
F(x,s)|:(x,s)\in \Omega\times[0,R_\theta]\}$. Hence $(u_{n})$ is
bounded in $W^{1,2}_{0}(\Omega)$ and, up to a subsequence, for
some $u_0\in W^{1,2}_0(\Omega)$, one has
\begin{equation}\label{fatos}
\begin{aligned}
&u_n\rightharpoonup u_0\,\,\,\, \textrm{in }\,\,W^{1,2}_0(\Omega),\\
&u_n\rightarrow u_0\,\,\,\, \textrm{in }\,\,L^p(\Omega)\,\,\,\,
\textrm{for }\,\, 1\leq p<\infty.
\end{aligned}
\end{equation}
In particular, $u_n(x)\rightarrow u_0(x)$ for almost every
$x\in\Omega$ and by (\ref{limitada}) it also follows that
$\int_{\Omega}|f(x,u_{n})u_{n}| \textrm{d}x$ is bounded. Thus, we
can apply Lemma \ref{djairo} to conclude that
\[
\int_{\Omega}f(x,u_{n}) \ \textrm{d}x\rightarrow
\int_{\Omega}f(x,u_0)  \ \textrm{d}x
\]
and therefore using $(f_1)$ and generalized Lebesgue dominated
convergence theorem, we can see that
\begin{equation}\label{ggabriel}
\displaystyle\int_{\Omega}F(x,u_{n}) \ \textrm{d}x\rightarrow
\displaystyle\int_{\Omega}F(x,u_0)  \ \textrm{d}x.
\end{equation}
At this point, we affirm that $u_0\ne 0$. In fact, suppose by
contradiction that $u_0\equiv 0$. Hence,
$\int_{\Omega}F(x,u_{n})\textrm{d}x\rightarrow 0$ and so
$$
\frac{1}{2}M(\|u_n\|^2)\rightarrow
c_*<\frac{1}{2}M(4\pi/\alpha_0).
$$
Thus, there exist $n_0\in \mathbb{N}$ and $\beta>0$ such that
$\alpha_0\|u_n\|^2<\beta<4\pi$ for all $n\geq n_0$. Now, choose
$q>1$ close to 1 and $\alpha>\alpha_0$ close to $\alpha_0$ so that
we still have $q\alpha\|u_n\|^2<\beta<4\pi$. From this and by
using \eqref{origem}, $(c)_{\alpha_0}$, H\"older inequality,
\eqref{PTM1} and \eqref{fatos} we get
$$
\begin{aligned}
\left|\int_{\Omega}f(x,u_{n})u_n \textrm{d}x\right|&\leq
C_1\int_{\Omega}|u_n|^2\
\textrm{d}x+C_2\int_{\Omega}|u_n|\exp(\alpha u_n^2)\ \textrm{d}x\\
&\leq
C_1\|u_n\|_2^2+C_2\|u_n\|_{{\frac{q}{q-1}}}\left(\int_{\Omega}\exp\left[q\alpha\|u_n\|^2\Large(u_n/\|u_n\|)^2\right]
\textrm{d}x\right)^{\frac{1}{q}}\\
&\leq C_1\|u_n\|_2^2+C_3\|u_n\|_{{\frac{q}{q-1}}}\,\longrightarrow
0\quad \textrm{as}\quad n\rightarrow\infty.
\end{aligned}
$$
Hence, since
$$
\langle I'(u_n),u_n\rangle=
m(\|u_n\|^2)\|u_n\|^2-\int_\Omega f(x,u_{n})u_n \textrm{d}x
$$
and $\langle I'(u_n),u_n\rangle\rightarrow 0$ it follows that
$m(\|u_n\|^2)\|u_n\|^2\rightarrow 0$. Consequently by $(M_1)$
$\|u_n\|^2\rightarrow 0$ and therefore $I(u_n)\rightarrow 0$, what
is absurd and thus we must have $u_0\ne 0$. Next, we will make
some assertions.\\

\noindent \textbf{Assertion 1.}\quad $u_0>0$ in $\Omega$.\\

\noindent \textit{Proof:} As $(u_n)$ is bounded, up to a
subsequence, $\|u_n\|\rightarrow \rho_0> 0$. Moreover, condition
$I'(u_n)\rightarrow 0$ implies that
\begin{equation}\label{solution}
m(\rho_0^2)\int_\Omega\nabla u_0\nabla v\ \textrm{d}x=\int_\Omega
f(x,u_0)v\ \textrm{d}x,\quad\forall\ v\in W^{1,2}_0(\Omega).
\end{equation}
Taking $v=-u_0^-$, where $u^{\pm}=\max\{\pm u,0\}$, it follows
that $\|u^-\|^2=0$ and so $u=u^+\geq 0$. Using the growth of $f$
and Trudinger-Moser inequality, $f(\cdot,u_0)\in L^p(\Omega)$ for
all $1\leq p<\infty$ and therefore by elliptic regularity $u_0\in
W^{2,p}(\Omega)$ for all $1\leq p<\infty$. Hence, by virtue of
Sobolev imbedding $u_0\in C^{1,\gamma}(\overline{\Omega})$. Now,
if we define $\Omega_0:=\{x\in \Omega:u_0(x)=0\}$ and we suppose
$\Omega_0\ne \emptyset$ then since $f(x,s)\geq 0$ and by applying
a Harnark inequality (see Theorem 8.20 in \cite{livro Trudinger})
we can conclude that $\Omega_0$ is an open and closed of $\Omega$.
The connectedness of $\Omega$ implies $\Omega_0=\Omega$ and so
$u_0\equiv 0$, which is a contradiction.
Thus, we must have $\Omega_0=\emptyset$, i.e., $u_0>0$ in $\Omega$.\\

\noindent \textbf{Assertion 2.}\quad $m(\|u_0\|^2)\|u_0\|^2\geq
\int_\Omega f(x,u_0)u_0\ \textrm{d}x$.\\

\noindent \textit{Proof:} Suppose by contradiction that
$m(\|u_0\|^2)\|u_0\|^2< \int_\Omega f(x,u_0)u_0\ \textrm{d}x$,
that is, $\langle I'(u_0),u_0\rangle<0$. Using $(M_1)$,
\eqref{origem} and Sobolev imbedding, we can see that $\langle
I'(tu_0),u_0\rangle>0$ for $t$ sufficiently small. Thus, there
exists $\sigma\in (0,1)$ such that $\langle I'(\sigma
u_0),u_0\rangle=0$, i.e., $\sigma u_0\in \mathcal{N}$. Thus,
according to $(\widehat{M}_3)$, Lemma \ref{crescente},
semicontinuity of norm and Fatou Lemma we obtain
$$
\begin{aligned}
c_*\leq b\leq I(\sigma u_0)&=I(\sigma u_0)-\frac{1}{4}\langle
I'(\sigma
u_0),\sigma u_0\rangle\\
&=\frac{1}{2}M(\|\sigma u_0\|^2)-\frac{1}{4}m(\|\sigma
u_0\|^2)\|\sigma u_0\|^2\\
&\quad +\frac{1}{4}\int_\Omega[f(x,\sigma
u_0)\sigma u_0-4F(x,\sigma u_0)]\ \textrm{d}x\\
&<\frac{1}{2}M(\| u_0\|^2)-\frac{1}{4}m(\|u_0\|^2)\|u_0\|^2\\
&\quad +\frac{1}{4}\int_\Omega[f(x,u_0) u_0-4F(x,u_0)]\ \textrm{d}x\\
&\leq \liminf_{n\rightarrow\infty}\left[\frac{1}{2}M(\| u_n\|^2)-\frac{1}{4}m(\|u_n\|^2)\|u_n\|^2\right]\\
&\quad +\liminf_{n\rightarrow\infty}\left[\frac{1}{4}\int_\Omega(f(x,u_n) u_n-4F(x,u_n))\ \textrm{d}x\right]\\
&\leq \lim_{n\rightarrow\infty}\left[I( u_n)-\frac{1}{4}\langle
I'(u_n), u_n\rangle\right]=c_*,
\end{aligned}
$$
which is absurd and the assertion is proved.\\

\noindent \textbf{Assertion 3.}\quad $I(u_0)=c_*$.\\

\noindent \textit{Proof:} By using \eqref{ggabriel} and
semicontinuity of norm, we have $I(u_0)\leq c_*$. We are going to
show that the case $I(u_0)< c_*$ can not occur. Indeed, if
$I(u_0)< c_*$ then $\|u_0\|^2<\rho_0^2$. Moreover,
\begin{equation}\label{conv funcional}
\frac{1}{2}M(\rho_0^2)=\lim_{n\rightarrow
\infty}\frac{1}{2}M(\|u_n\|^2)=c_*+\int_{\Omega}F(x,u_0)\textrm{d}x,
\end{equation}
which implies
$\rho_0^2=M^{-1}(2c_*+2\int_{\Omega}F(x,u_0)\textrm{d}x)$. Next,
defining $v_n=u_n/\|u_n\|$ and $v_0=u_0/\rho_0$, we have
$v_n\rightharpoonup v_0$ in $W^{1,2}_0(\Omega)$ and $\|v_0\|<1$.
Thus, by \eqref{Lions-1}
\begin{equation}\label{por Lions}
\sup_{n\in \mathbb{N}}\int_{\Omega}\exp(pv_n^2)\
dx<\infty,\quad\forall\ p<\frac{4\pi}{1-\|v_0\|^2}.
\end{equation}
On the other hand, by Assertion 2, \eqref{conseq1} and Lemma
\ref{crescente} one has
$$
I(u_0)\geq
\frac{1}{2}M(\|u_0\|^2)-\frac{1}{4}m(\|u_0\|^2)\|u_0\|^2
+\frac{1}{4}\int_\Omega[f(x,u_0)u_0-4F(x,u_0)]\textrm{d}x\geq 0.
$$
Using this information together with Lemma \ref{nivel} and the
equality
$$
2c_*-2I(u_0)=M(\rho_0^2)-M(\|u_0\|^2),
$$
where we have used \eqref{conv funcional}, we get
$$
M(\rho_0^2)\leq
2c_*+M(\|u_0\|^2)<M\left(\frac{4\pi}{\alpha_0}\right)+M(\|u_0\|^2)
$$
and therefore by
$(M_1)$
\begin{equation}\label{manga}
\rho_0^2<M^{-1}\left[M\left(\frac{4\pi}{\alpha_0}\right)
+M(\|u_0\|^2)\right]\leq\frac{4\pi}{\alpha_0}+\|u_0\|^2.
\end{equation}
Since
$$
\rho_0^2=\frac{\rho_0^2-\|u_0\|^2}{1-\|v_0\|^2},
$$
from \eqref{manga} it follows that
$$
\rho_0^2<\frac{\frac{4\pi}{\alpha_0}}{1-\|v_0\|^2}.
$$
Thus, there exists $\beta>0$ such that $\alpha_0\|u_n\|^2<\beta<
4\pi/(1-\|v_0\|^2)$ for $n$ sufficiently large. For $q>1$ close to
1 and $\alpha>\alpha_0$ close to $\alpha_0$ we still have
$q\alpha\|u_n\|^2\leq\beta< 4\pi/(1-\|v_0\|^2)$ and invoking
\eqref{por Lions}, for some $C>0$ and $n$ large enough, we
concluded that
$$
\int_{\Omega}\exp(q\alpha u_n^2)\ dx\leq \int_{\Omega}\exp(\beta
v_n^2)\ dx\leq C.
$$
Hence, using \eqref{origem}, $(c)_{\alpha_0}$, H\"older
inequality, \eqref{PTM1} and \eqref{fatos} we reach
$$
\begin{aligned}
\left|\int_{\Omega}f(x,u_{n})(u_n-u_0) \textrm{d}x\right|&\leq
C_1\int_{\Omega}|u_n-u|^2\
\textrm{d}x+C_2\int_{\Omega}|u_n-u_0|\exp(\alpha u_n^2)\ \textrm{d}x\\
&\leq
C_1\|u_n-u_0\|_2^2+C_3\|u_n-u_0\|_{{\frac{q}{q-1}}}\,\longrightarrow
0\quad \textrm{as}\quad n\rightarrow\infty.
\end{aligned}
$$
Since $\langle I'(u_n),u_n-u_0\rangle\rightarrow 0$, it follows
that $m(\|u_n\|^2)\int_{\Omega}\nabla u_n(\nabla u_n-\nabla
u_0)\textrm{d}x\rightarrow 0$. On the other hand,
$$
\begin{aligned}
m(\|u_n\|^2)\int_{\Omega}\nabla u_n(\nabla u_n-\nabla
u_0)\textrm{d}x=&\ m(\|u_n\|^2)\|u_n\|^2-
m(\|u_n\|^2)\int_{\Omega}\nabla u_n\nabla u_0\
\textrm{d}x\\
&\longrightarrow m(\rho_0^2)\rho_0^2-m(\rho_0^2)\|u_0\|^2,
\end{aligned}
$$
which implies that $\rho_0=\|u_0\|$. Thus, $\|u_n\|\rightarrow
\|u_0\|$ and so $u_n\rightarrow u_0$. In view of the continuity of
$I$, we must have $I(u_0)=c_*$ what is an absurd and the assertion
is proved.\\

\noindent\textit{Finalizing the proof of Theorem \ref{teorema1}:}
By Assertion 3 and \eqref{conv funcional} we can see that
$M(\rho_0^2)=M(\|u_0\|^2)$ which shows that $\rho_0^2=\|u_0\|^2$.
Hence, by \eqref{solution} we have
$$
m(\|u_0\|^2)\int_\Omega\nabla u_0\nabla v\ \textrm{d}x=\int_\Omega
f(x,u_0)v\ \textrm{d}x,\quad\forall\ v\in W^{1,2}_0(\Omega),
$$
that is, $u_0$ is a solution of \eqref{P} satisfying $I(u_0)=c_*$
and according to Remark \ref{ground1} and Assertion 1 the proof of
our main result is complete.
\end{proof}

\begin{Remark}
As a matter of fact, by a slight modification of the previous
proof, we can prove that the functional $I$ satisfies the
Palais-Smale condition at any level $c\in \left(-\infty,
\frac{1}{2}M(4\pi/\alpha_0)\right)$.
\end{Remark}

\begin{Remark}
In line with Remark \ref{constante}, we note that if $m(t)\equiv
K$ then automatically the equality holds in the Assertion 1, and
thus we would not need the condition $(f_2)$ to prove Assertion 1.
\end{Remark}

\noindent \textbf{Acknowledgment} Part of this work was done while
the second author was visiting the Universidade Federal do Par\'a
- UFPA. He would like to thank professor Giovany M. Figueiredo for
his hospitality.

\end{document}